\documentclass[11pt]{article}
\usepackage{amsfonts}
\usepackage{amssymb}
\usepackage{amsthm}
\usepackage{amsmath}

\usepackage{bm}
\usepackage{amscd} 

\usepackage{hyperref}

\numberwithin{equation}{section} \DeclareMathSizes{2}{10}{12}{13}
\newtheorem{thm}{Proposition}[section]

\newtheorem{Thm}[thm]{Theorem}

\newtheorem{lem}[thm]{Lemma}
\newtheorem{defn}[thm]{Definition}

\newcommand{\leftsub}[2]{{\vphantom{#2}}_{#1}{#2}}

\title{
The $q$-analog of higher order Hochschild homology and the Lie derivative}
\author{Abhishek Banerjee
}
\date{ }

\begin{document}

\maketitle

\centerline{\emph{Dept. of Mathematics, Indian Institute of Science, Bengaluru - 560012, Karnataka, India. }}
\centerline{\emph{Email: abhishekbanerjee1313@gmail.com}}

\medskip
\begin{abstract}  Let $A$ be a commutative algebra over $\mathbb C$. Given a pointed simplicial finite set $Y$ and $q\in \mathbb C$ a primitive $N$-th root of unity, we define the $q$-Hochschild
homology groups $\{\leftsub{q}{HH}_n^Y(A)\}_{n\geq 0}$ of $A$ of order $Y$. When $D$ is a derivation on $A$, we construct the corresponding Lie derivative on the groups $\{\leftsub{q}{HH}_n^Y(A)\}_{n\geq 0}$. We also define the Lie derivative on $\{\leftsub{q}{HH}_n^Y(A)\}_{n\geq 0}$
for a higher derivation $\{D_n\}_{n\geq 0}$ on $A$. Finally, we describe the morphisms induced
on the bivariant $q$-Hochschild cohomology groups $\{\leftsub{q}{HH}^n_Y(A,A)\}_{n\in\mathbb Z}$ 
of order $Y$ by a derivation $D$ on $A$. 
\end{abstract}
 
 \medskip
  {\bf MSC(2010) Subj Classification:} 16W25  
  
  Keywords: Hochschild homology, higher derivations, Lie derivative

 \section{Introduction}
 
 \medskip
 
 Let $A$ be a commutative algebra over $\mathbb C$. Then, it is well known (see, for instance,
 \cite[$\S$ 4.1]{Lod}) that a derivation $D:A\longrightarrow A$ induces morphisms
 \begin{equation}
 L_D^n:HH_n(A)\longrightarrow HH_n(A) \qquad \forall \textrm{ }n\geq 0
 \end{equation} on the Hochschild homology groups of the algebra $A$. The morphisms
 $L_D^n$, $n\geq 0$ play the role of the Lie derivative in noncommutative geometry. For more 
 on these morphisms and for general properties of Hochschild homology, we refer the reader
 to \cite{Lod}. Further, for any pointed simplicial finite set $Y$, Pirashvili \cite{TP} has introduced
 the Hochschild homology groups $\{HH_n^Y(A)\}_{n\geq 0}$ of $A$ of order $Y$ (see also
 Loday \cite{Lod2}). In particular, when $Y=S^1$ is the simplicial circle, the groups 
 $\{HH_n^{S^1}(A)\}_{n\geq 0}$ reduce to the usual Hochschild homology groups of the algebra
 $A$. Let $q\in \mathbb C$ be a primitive $N$-th root of unity. The purpose of this paper is to introduce the $q$-analogues 
 $\{\leftsub{q}{HH}_n^Y(A)\}_{n\geq 0}$ of these higher order Hochschild homology groups
 and study the morphisms induced on them by derivations on $A$. 
 
 \medskip
 More precisely, let $\Gamma$ denote the category whose objects are the finite sets $[n]=\{0,1,2,..,n\}$, $n\geq 0$ with basepoint $0\in [n]$. Then, given the algebra $A$, we can define a functor $\mathcal L(A)$ from $\Gamma$ to the category $Vect$ of complex vector spaces that takes $[n]$ to
 $A\otimes A^{\otimes n}$ (see Section 2 for details). Then, we can prolong $\mathcal L(A)$ by means
 of colimits to a functor $\mathcal L(A):Fin_*\longrightarrow Vect$ from the category $Fin_*$ of all finite sets with basepoint. Given a pointed simplicial finite set $Y$, i.e., a functor $Y:\Delta^{op}
 \longrightarrow Fin_*$ ($\Delta^{op}$ being the simplex category), we now have a simplicial
 vector space
 \begin{equation}
 \begin{CD}
 \mathcal L^Y(A):\Delta^{op}@>Y>> Fin_*@>\mathcal L(A)>> Vect
 \end{CD}
 \end{equation} Let $d^j_i:\mathcal L^Y(A)_i\longrightarrow 
 \mathcal L^Y(A)_{i-1}$, $0\leq j\leq i$, $i\geq 0$ be the face maps of the simplicial vector space $\mathcal L^Y(A)$. We then construct the ``$q$-Hochschild differentials'':
 \begin{equation}
 \leftsub{q}{b}_i:\mathcal L^Y(A)_{i}\longrightarrow \mathcal L^Y(A)_{i-1}\qquad 
 \leftsub{q}{b}_i:=\sum_{j=0}^iq^jd_i^j
 \end{equation} Since $q\in \mathbb C$ is a primitive $N$-th root of unity, it follows that
 $\leftsub{q}{b}^N=0$, i.e., $(\mathcal L^Y(A),\leftsub{q}{b})$ is an $N$-complex in the sense
 of Kapranov \cite{Kap}. We now define the $q$-Hochschild homology groups $\{\leftsub{q}{HH}_n^Y(A)\}_{n\geq 0}$ of $A$ of order $Y$ to be the homology objects of the $N$-complex
 $(\mathcal L^Y(A),\leftsub{q}{b})$ (see Definitions \ref{Def00} and \ref{Def01}). When $q=-1$ (and hence $N=2$), $\leftsub{q}{b}$ reduces to the usual differential on the chain complex associated to the simplicial
 vector space $\mathcal L^Y(A)$ and we have $\leftsub{(-1)}{HH}^Y_n(A)=HH_n^Y(A)$, $\forall$
 $n\geq 0$. Then, the main result of Section 2 is as follows. 
 
 \medskip
 \begin{Thm}\label{Thm1} Let $D:A\longrightarrow A$ be a derivation on $A$. Then, for each $n\geq 0$, the 
 derivation $D$ induces a morphism $L_D^{Y,n}:\leftsub{q}{HH}_n^Y(A)\longrightarrow 
 \leftsub{q}{HH}_n^Y(A)$ of $q$-Hochschild homology groups of order $Y$. Additionally, if
 $\mathcal H=\mathcal U(Der(A))$ is the universal enveloping algebra of the Lie algebra
 $Der(A)$ of derivations on $A$, each $\leftsub{q}{HH}_n^Y(A)$ is a left $\mathcal H$-module,
 i.e., for any element $h\in \mathcal H$, there exist morphisms 
 $L_h^{Y,n}:\leftsub{q}{HH}_n^Y(A)\longrightarrow 
 \leftsub{q}{HH}_n^Y(A)$ of $q$-Hochschild homology groups of order $Y$.
 \end{Thm}
 
 Thereafter, we consider a higher derivation $D=\{D_n\}_{n\geq 0}$ on $A$. We recall that a higher (or Hasse-Schmidt) derivation $D=\{D_n\}_{n\geq 0}$ on $A$ is a sequence of
 linear maps $D_n:A\longrightarrow A$ satisfying the following relation (see, for example, \cite{Mat86}):
 \begin{equation}
 D_n(a\cdot a')=\sum_{i=0}^nD_i(a)\cdot D_{n-i}(a')\qquad \forall \textrm{ }a,a'\in A, n\geq 0
 \end{equation} In this paper, we restrict ourselves to normalized higher derivations, i.e.,
 higher derivations $D=\{D_n\}_{n\geq 0}$ such that $D_0=1$. Then, in Section 3, we construct the Lie derivative on the $q$-Hochschild homology groups of order $Y$ corresponding to a higher derivation $D=\{D_n\}_{n\geq 0}$. 
 
 \medskip
 \begin{Thm}\label{Thm2} Let $D=\{D_n\}_{n\geq 0}$ be a normalized higher derivation on $A$. Then, for each $k\geq 0$, we have an induced morphism
$L_D^{Y,k}:\leftsub{q}{HH}_*^Y(A)=\underset{n=0}{\overset{\infty}{\bigoplus}}\textrm{ }\leftsub{q}{HH}_n^Y(A)\longrightarrow 
\leftsub{q}{HH}_*^Y(A) =\underset{n=0}{\overset{\infty}{\bigoplus}}\textrm{ }\leftsub{q}{HH}_n^Y(A)$ on the $q$-Hochschild homology groups of $A$ of order $Y$. 
 \end{Thm} 
 
 \medskip
 Further, in \cite{Mirza}, Mirzavaziri has provided a characterization of normalized higher derivations on algebras over $\mathbb C$ from which it follows that if   $D=\{D_k\}_{k\geq 0}$ is a higher derivation on
 $A$, each $D_k$ is an element of the Hopf algebra $\mathcal H=\mathcal U(Der(A))$ (see \eqref{3.7} for details). It follows
 therefore from Theorem \ref{Thm1} that for each $k$, the element $D_k\in \mathcal H$ induces a 
 morphism $L_{D_k}^Y:\leftsub{q}{HH}_*^Y(A)\longrightarrow \leftsub{q}{HH}_*^Y(A)$. Then, in Section 3, we prove the following result. 
  
 \begin{Thm} Let $D=\{D_k\}_{k\geq 0}$ be a normalized higher derivation on $A$. Then, for each
 $k\geq 1$, we have $L^Y_{D_k}=L_D^{Y,k}$ as an endomorphism of $\leftsub{q}{HH}_*^Y(A)=\underset{n=0}{\overset{\infty}{\bigoplus}}\textrm{ }\leftsub{q}{HH}_n^Y(A)$. 
 \end{Thm}
 
 \medskip
 In Section 4, we start by defining bivariant $q$-Hochschild cohomology groups $\{
 \leftsub{q}{HH}^n_Y(A,A)\}_{n\in \mathbb Z}$ of order $Y$. For this we 
  consider  the modules $\underline{Hom}(\mathcal L^Y(A),\mathcal L^Y(A))_n$, $n\in 
  \mathbb Z$ where an element $f\in  \underline{Hom}(\mathcal L^Y(A),\mathcal L^Y(A))_n$ is 
  given by a family of morphisms $f=\{f_i:\mathcal L^Y(A)_i\longrightarrow 
  \mathcal L^Y(A)_{i+n}\}_{i\in \mathbb Z}$. Further, we define a differential 
  $\leftsub{q}{\partial}_n: \underline{Hom}(\mathcal L^Y(A),\mathcal L^Y(A))_n
  \longrightarrow \underline{Hom}(\mathcal L^Y(A),\mathcal L^Y(A))_{n-1}$ (see 
  Definition \ref{Df4.1}). Then, $(\underline{Hom}(\mathcal L^Y(A),\mathcal L^Y(A)),
  \leftsub{q}{\partial})$ is an $N$-complex and we let $\leftsub{q}{HH}^n_Y(A,A)$
  be the $(-n)$-th homology object of $(\underline{Hom}(\mathcal L^Y(A),\mathcal L^Y(A)),
  \leftsub{q}{\partial})$. We end with the following result. 
  
  \begin{Thm} Let $D:A\longrightarrow A$ be a derivation on $A$. Then, for each $n\in 
  \mathbb Z$, the derivation $D$ induces a morphism
  $\underline{L}_D^{Y,n}: \leftsub{q}{HH}_Y^n(A,A)\longrightarrow  \leftsub{q}{HH}_Y^n(A,A)
$ on the bivariant $q$-Hochschild cohomology groups of order $Y$. 
  \end{Thm}

 \medskip
 \section{Lie Derivative on higher order $q$-Hochschild homology}

 \medskip
 Let $Vect$ denote the category of vector spaces over $\mathbb C$. Let $A$ be a   commutative
 $\mathbb C$-algebra.  We recall here the definition of higher order Hochschild homology groups of a commutative algebra
 $A$ as introduced by Pirashvili \cite{TP} (see also Loday \cite{Lod2}). Let $\Gamma$ denote the category whose objects are the pointed
 sets $[n]=\{0,1,2,...,n\}$ with $0\in [n]$ as base point for each $n\geq 0$. Then, a morphism $\phi :[m]\longrightarrow [n]$ 
 in $\Gamma$ is a map $\phi:\{0,1,2,...,m\}\longrightarrow \{0,1,2,...,n\}$ of sets such that $\phi(0)=0$.  We now define
 a functor:
 \begin{equation}\label{2.1}
\mathcal L(A):\Gamma\longrightarrow Vect \qquad [n]\mapsto A\otimes A^{\otimes n}
 \end{equation} Given a morphism $\phi:[m]\longrightarrow [n]$ in $\Gamma$, we have an induced map in 
 $Vect$:
 \begin{equation}\label{2.2}
 \begin{array}{c}
 \mathcal L(A)(\phi):A\otimes A^{\otimes m}\longrightarrow A\otimes A^{\otimes n} 
\\
\mathcal L(A)(\phi)(a_0\otimes a_1\otimes ....\otimes a_m) = (b_0\otimes b_1\otimes ...\otimes b_n) \qquad b_j:=\underset{\phi(i)=j}{\prod}a_i 
\\ \end{array}
 \end{equation} We now consider the category $Fin_*$ of finite pointed sets. There is a natural inclusion $
 \Gamma\hookrightarrow Fin_*$ of categories. Then, $\mathcal L(A):\Gamma
 \longrightarrow Vect$ can be extended to a functor $\mathcal L(A):Fin_*\longrightarrow Vect$ by setting:
 \begin{equation}\label{2.3}
 \mathcal L(A):Fin_*\longrightarrow Vect \qquad T\mapsto \underset{\Gamma\ni T'\longrightarrow T}{colim}\textrm{ }\mathcal L(A)(T') 
 \end{equation} where the colimit in \eqref{2.3} is taken over all morphisms $T'\longrightarrow T$ in $Fin_*$ such that
 $T'\in \Gamma$. Let $\Delta$ be the simplex category, i.e., the category whose objects are sets $[n]=\{0,1,2,...,n\}$, $n\geq 0$ and whose
 morphisms are order preserving maps. Then, given a pointed simplicial finite  set $Y$ corresponding to a functor $Y:\Delta^{op}
 \longrightarrow Fin_*$, we have a simplicial vector space $\mathcal L^Y(A)$ determined by the composition of functors:
 \begin{equation}\label{2.4}
 \begin{CD}
\mathcal L^Y(A): \Delta^{op}\overset{Y}{\longrightarrow}Fin_* @>\mathcal L(A)>> Vect \\
 \end{CD}
 \end{equation}   For any $n\geq 0$, let $HH_n^Y(A)$ denote the $n$-th homology group of the chain complex 
 associated to the simplicial vector space $\mathcal L^Y(A)$. Following Pirashvili \cite{TP}, when $Y=S^p$ ($S^p$ being the sphere of dimension $p\geq 1$), we say
 that the homology groups $\{HH_n^{S^p}(A)\}_{n\geq 0}$ are the Hochschild homology groups of $A$ of order
 $p$. When $p=1$, i.e., $Y=S^1$ is the simplicial circle, the Hochschild homology groups $\{HH_n^{S^1}(A)\}_{n\geq 0}$
 are identical to the usual Hochschild homology groups of $A$. 
 
 \medskip
Our objective is to introduce a $q$-analog of the groups $HH_*^Y(A)$, where $q\in \mathbb C$ is a primitive $N$-th root
of unity.  For this, we consider the  face maps $d_{n}^i:\mathcal L^Y(A)_n
\longrightarrow \mathcal L^Y(A)_{n-1}$, $0\leq i\leq n$, $n\geq 0$, of the simplicial vector space $\mathcal L^Y(A)$ defined
in \eqref{2.4}. We set:
\begin{equation}\label{2.5cxw}
\leftsub{q}{b}_n:\mathcal L^Y(A)_n\longrightarrow \mathcal L^Y(A)_{n-1}\qquad \leftsub{q}{b}_n:=\sum_{i=0}^nq^id_{n}^i
\end{equation} For the sake of convenience, we will often write $\leftsub{q}b_n$ simply as $\leftsub{q}{b}$.  Then, it is well known that the morphism $\leftsub{q}{b}$ satisfies $\leftsub{q}{b}^N=0$ (this is true in general for any simplicial vector space; see, for instance, Kapranov \cite[Proposition 0.2]{Kap}). In particular, if $q=-1$, i.e., $N=2$, we have $\leftsub{(-1)}{b}^2=0$ and $\leftsub{(-1)}{b}$ is the 
standard differential on the chain complex corresponding to the simplicial vector space
$\mathcal L^Y(A)$. In general, the pair $(\mathcal L^Y(A),\leftsub{q}{b})$, i.e., the simplicial vector space $\mathcal L^Y(A)$ equipped with
the morphism $\leftsub{q}{b}$ is an ``$N$-complex'' in the sense defined below. 

\medskip
\begin{defn}\label{Def00} (see \cite[$\S$ 2]{MDV1} and \cite[Definition 0.1]{Kap}) Let $\mathcal A$
be an abelian category and $N\geq 2$ a positive integer. An $N$-complex 
in $\mathcal A$ is a sequence of objects and morphisms of $\mathcal A$
\begin{equation}\label{2.6}
C_*=\{\dots \longrightarrow C_1\overset{b_1}{\longrightarrow} C_0\overset{b_0}{\longrightarrow} C_{-1}\longrightarrow \dots \}
\end{equation} such that the composition of any $N$ consecutive morphisms in \eqref{2.6}
is $0$. For any $n\in \mathbb Z$, the homology object $H_{\{n\}}(C_*,b)$ of the $N$-complex
$(C_*,b)$ is defined as:
\begin{equation}
H_{\{n\}}(C_*,b):=\underset{i=1}{\overset{N-1}{\bigoplus}} H_{\{i,n\}}(C_*,b) \qquad H_{\{i,n\}}(C_*,b):=\frac{Ker(b^i:C_n\longrightarrow C_{n-i})}{Im(b^{N-i}:C_{N-i+n}\longrightarrow C_n)}
\end{equation}

\end{defn}

\medskip
\begin{defn}\label{Def01} Let $A$ be a commutative algebra over $\mathbb C$ and let $Y$ be a pointed simplicial finite set. Let $q\in 
\mathbb C$ be a primitive $N$-th root of unity. Then, the $q$-Hochschild homology groups $\leftsub{q}{HH}_n^Y(A)$, $n\geq 0$ of $A$ of order $Y$ are defined to be the homology objects of 
the $N$-complex $(\mathcal L^Y(A),\leftsub{q}{b})$ associated to the simplicial vector space $\mathcal L^Y(A)$; in other words, we define:
\begin{equation}
\leftsub{q}{HH}_n^Y(A):=H_{\{n\}}(\mathcal L^Y(A),\leftsub{q}{b})
\end{equation}

\end{defn}

\medskip
As with the ordinary Hochschild homology of an algebra (see, for instance, \cite[$\S$ 4.1]{Lod}), given a derivation $D:A\longrightarrow A$,  we want to construct the Lie derivative
$L_D^Y:{HH}_*^Y(A)\longrightarrow {HH}_*^Y(A)$ on the Hochschild homology of
order $Y$. For this, we start with the following lemma.

\medskip
\begin{lem}\label{lem1} Let $A$ be a commutative $\mathbb C$-algebra and let $D:A\longrightarrow A$ be a derivation on $A$. Then, the derivation
$D$ induces an endomorphism $L_D:\mathcal L(A)\longrightarrow \mathcal L(A)$ of the functor 
$\mathcal L(A):Fin_*\longrightarrow Vect$.

\end{lem}

\begin{proof} We first consider the functor $\mathcal L(A)$ restricted to the subcategory $\Gamma$ of
$Fin_*$, defined as in \eqref{2.1} and \eqref{2.2}:
\begin{equation}
\mathcal L(A):\Gamma\longrightarrow Vect \qquad [n]\mapsto A\otimes A^{\otimes n}
 \end{equation} Given the derivation $D$ on $A$, we define morphisms (for all $n\geq 0$):
\begin{equation}\label{2.7cq}
L_D([n]):\mathcal L(A)([n])\longrightarrow \mathcal L(A)([n])\qquad (a_0\otimes a_1\otimes ...\otimes a_n)\mapsto \sum_{i=0}^n(a_0\otimes a_1\otimes...\otimes D(a_i)\otimes...\otimes a_n)
\end{equation} Further, for any morphism $\phi:[m]\longrightarrow [n]$ in $\Gamma$, we have, for any
$(a_0\otimes a_1\otimes ... \otimes a_m)\in A\otimes A^{\otimes m}$:
\begin{equation*}
\begin{array}{ll}
L_D([n])\circ \mathcal L(A)(\phi)(a_0\otimes a_1\otimes ... \otimes a_m)  &= L_D([n])\left(\underset{j=0}{\overset{n}{\bigotimes}}
\textrm{ }\underset{\phi(i)=j}{\prod}a_i\right)
\\ 
&= \underset{k=0}{\overset{n}{\sum}}\left(\underset{j=0}{\overset{k-1}{\bigotimes}}
\textrm{ }\underset{\phi(i)=j}{\prod}a_i\right)\otimes D\left(\underset{\phi(i)=k}{\prod} a_i\right )\otimes \left(\underset{j=k+1}{\overset{n}{\bigotimes}}
\textrm{ }\underset{\phi(i)=j}{\prod}a_i\right)
\\
&= \underset{k=0}{\overset{n}{\sum}}\textrm{ } \underset{i\in \phi^{-1}(k)}{\sum} \mathcal L(A)(\phi)(a_0\otimes ...\otimes
D(a_i)\otimes ... \otimes a_m)\\
& =\underset{i=0}{\overset{m}{\sum}}
\mathcal L(A)(\phi)(a_0\otimes ...\otimes
D(a_i)\otimes ... \otimes a_m)\\
& =\mathcal L(A)(\phi)\circ L_D([m])(a_0\otimes a_1\otimes ... \otimes a_m) \\
\end{array}
\end{equation*} It follows that the derivation $D$ induces an endomorphism $L_D$ of the functor $\mathcal L(A):\Gamma
\longrightarrow Vect$. More generally, for any object $T\in Fin_*$ and a morphism $T'\longrightarrow T$ in $Fin_*$ such that
$T'\in \Gamma$, we have a morphism $L_D(T'):\mathcal L(A)(T')\longrightarrow \mathcal L(A)(T')$ as defined in 
\eqref{2.7cq}. By definition, we know that $\mathcal L(A)(T)=\underset{\Gamma\ni T'\longrightarrow T}{colim}
\mathcal L(A)(T')$ and hence we have an induced morphism
\begin{equation}\label{2.8cq}
L_D(T):\mathcal L(A)(T)=\underset{\Gamma\ni T'\longrightarrow T}{colim}
\mathcal L(A)(T')\longrightarrow \mathcal L(A)(T)=\underset{\Gamma\ni T'\longrightarrow T}{colim}
\mathcal L(A)(T')
\end{equation} From \eqref{2.8cq} it follows that the derivation $D$ induces an endomorphism $L_D:
\mathcal L(A)\longrightarrow \mathcal L(A)$ of the functor $\mathcal L(A):Fin_*
\longrightarrow Vect$. This proves the claim. 
 
\end{proof}

\medskip

\begin{thm}\label{prop2.4r} Let $A$ be a commutative $\mathbb C$-algebra and let $D:A\longrightarrow A$ be a derivation on $A$. Let $Y$ be a pointed
simplicial finite set. Then, for each
$n\geq 0$, the derivation $D$ induces a morphism $L_D^{Y,n}:\leftsub{q}{HH}_n^Y(A)\longrightarrow 
\leftsub{q}{HH}_n^Y(A)$ of $q$-Hochschild
homology groups of order $Y$, where $q\in\mathbb C$ is a primitive $N$-th root of unity. 

\end{thm}

\begin{proof} From Lemma \ref{lem1}, we know that the derivation $D$ induces an endomorphism
$L_D:\mathcal L(A)\longrightarrow \mathcal L(A)$ of the functor $\mathcal L(A):
Fin_*\longrightarrow Vect$. Given the  pointed simplicial finite set $Y$, the endomorphism $L_D:\mathcal L(A)
\longrightarrow \mathcal L(A)$ of functors induces an endomorphism of the functor
\begin{equation}\label{2.9cq}
\begin{CD}
\mathcal L^Y(A):\Delta^{op}\overset{Y}{\longrightarrow}Fin_* @>\mathcal L(A)>> Vect \\
\end{CD}
\end{equation} From \eqref{2.9cq}, it follows  that we have an endomorphism $L_D^Y:
\mathcal L^Y(A)\longrightarrow \mathcal L^Y(A)$ of the simplicial vector space
$\mathcal L^Y(A)$. Hence, we have induced morphisms $L_D^{Y,n}:\leftsub{q}{HH}_n^Y(A)
\longrightarrow \leftsub{q}{HH}_n^Y(A)$ on the homology objects of the $N$-complex $(\mathcal L^Y(A),\leftsub{q}{b})$ associated
to the simplicial vector space $\mathcal L^Y(A)$ as in \eqref{2.5cxw}. 
\end{proof}

\medskip
We now let $Der(A)$ denote the vector space of all derivations on the commutative $\mathbb C$-algebra $A$. Then, 
$Der(A)$ is a Lie algebra, endowed with the Lie bracket $[D,D']:=D\circ D' - D'\circ D$, $\forall$ $D$, $D'\in Der(A)$. Let 
$\mathcal H:=\mathcal U(Der(A))$ denote the universal enveloping algebra of $Der(A)$. We will now show that for any
pointed simplicial finite set $Y$, the operators
$L_D^{Y,n}$, $D\in Der(A)$ on the $q$-Hochschild homology group of $A$ of order $Y$ make $\leftsub{q}{HH}^Y_n(A)$ into a module
over the Hopf algebra $\mathcal H=\mathcal U(Der(A))$. 

\medskip

\begin{lem}\label{lemKS}  Let $q\in \mathbb C$ be a primitive $N$-th root of unity. Let $A$ be a commutative $\mathbb C$-algebra and let $D$, $D'\in Der(A)$ be derivations on $A$. Let $Y$ be a
pointed simplicial finite set. Then, for each $n\geq 0$, the operators $L_D^{Y,n}$, $L_{D'}^{Y,n}:\leftsub{q}{HH}^Y_n(A)\longrightarrow 
\leftsub{q}{HH}^Y_n(A)$ satisfy $[L_D^{Y,n}, L_{D'}^{Y,n}]=L^{Y,n}_D\circ L^{Y,n}_{D'} - L^{Y,n}_{D'}\circ L^{Y,n}_D=L^{Y,n}_{[D,D']}$. 
\end{lem} 

\begin{proof} For $D$, $D'\in Der(A)$, we consider the respective endomorphisms $L_D$, $L_{D'}$ of the functor
$\mathcal L(A):\Gamma\longrightarrow Vect$. By definition, for any object $[n]\in \Gamma$, we have morphisms:
\begin{equation}\label{2.9KS}
\begin{array}{c}
L_D([n]):\mathcal L(A)([n])\longrightarrow \mathcal L(A)([n])\qquad (a_0\otimes  ...\otimes a_n)\mapsto \underset{i=0}{\overset{n}{\sum}}(a_0\otimes...\otimes D(a_i)\otimes...\otimes a_n) \\
L_{D'}([n]):\mathcal L(A)([n])\longrightarrow \mathcal L(A)([n])\qquad (a_0\otimes ...\otimes a_n)\mapsto \underset{i=0}{\overset{n}{\sum}}(a_0\otimes...\otimes D'(a_i)\otimes...\otimes a_n) \\
\end{array}
\end{equation} From \eqref{2.9KS}, it may be verified easily that we have
\begin{equation}
(L_D\circ L_{D'}- L_{D'}\circ L_D)([n])=L_{[D,D']}([n]):\mathcal L(A)([n])\longrightarrow \mathcal L(A)([n])\qquad \forall \textrm{ } n\geq 0
\end{equation}  and it follows that $L_D\circ L_{D'}- L_{D'}\circ L_D=L_{[D,D']}$ as endomorphisms of the functor
$\mathcal L(A):\Gamma\longrightarrow Vect$. More generally, for any object $T\in Fin_*$, we have 
$\mathcal L(A)(T)=\underset{\Gamma\ni T'\longrightarrow T}{colim}\textrm{ }\mathcal L(A)(T')$ and hence
$L_D\circ L_{D'}- L_{D'}\circ L_D=L_{[D,D']}$ as endomorphisms of the functor
$\mathcal L(A):Fin_* \longrightarrow Vect$. Finally, considering the composition of $\mathcal L(A):Fin_*
\longrightarrow Vect$  with the functor $Y:\Delta^{op}\longrightarrow Fin_*$ corresponding
to the pointed simplicial finite set $Y$, it follows that $L_D^Y\circ L_{D'}^Y- L_{D'}^Y\circ L_D^Y=L^Y_{[D,D']}$ as endomorphisms of the functor $\mathcal L^Y(A):\Delta^{op}\longrightarrow Vect$. Hence, we have 
$[L_D^{Y,n}, L_{D'}^{Y,n}]=L^{Y,n}_D\circ L^{Y,n}_{D'} - L^{Y,n}_{D'}\circ L^{Y,n}_D=L^{Y,n}_{[D,D']}$ on the homology objects
$\leftsub{q}{HH}_n^Y(A)$, $n\geq 0$ of the $N$-complex $(\mathcal L^Y(A),\leftsub{q}{b})$ associated to the simplicial vector space
$\mathcal L^Y(A):\Delta^{op}\longrightarrow Vect$  as in \eqref{2.5cxw}. 
\end{proof}

\medskip
\begin{thm}\label{prop2.6}  Let $q\in \mathbb C$ be a primitive $N$-th root of unity.  Let $A$ be a commutative algebra over $\mathbb C$ and let $Der(A)$ denote the Lie algebra of derivations
on $A$. Let $\mathcal H=\mathcal U(Der(A))$ denote the universal enveloping algebra of $Der(A)$. Then, for any
pointed simplicial finite set $Y$ and any $n\geq 0$, the $q$-Hochschild homology group $\leftsub{q}{HH}_n^Y(A)$ of order $Y$ is a left module over
the Hopf algebra $\mathcal H$.
\end{thm}

\begin{proof} From Lemma \ref{lemKS}, it follows that   $Der(A)$ has a Lie algebra action on each $\leftsub{q}{HH}_n^Y(A)$, i.e., $[L_D^{Y,n}, L_{D'}^{Y,n}]=L^{Y,n}_D\circ L^{Y,n}_{D'} - L^{Y,n}_{D'}\circ L^{Y,n}_D=L^{Y,n}_{[D,D']}$ for any
$D$, $D'\in Der(A)$. Since $\mathcal H$ is the universal enveloping algebra of $Der(A)$, it follows that this Lie
algebra action of $Der(A)$ on $\leftsub{q}{HH}_n^Y(A)$  makes $\leftsub{q}{HH}^Y_n(A)$ into a left $\mathcal H$-module.

\end{proof}

\medskip
\section{Higher derivations and the Lie derivative}

\medskip
As before, we work with a commutative algebra $A$ over $\mathbb C$, a pointed simplicial finite set $Y$ and $q\in \mathbb C$ a primitive $N$-th root of unity. In this section, we will describe the Lie derivative on the $q$-Hochschild homology groups $\leftsub{q}{HH}_*^Y(A)$ corresponding to a higher derivation $D$ on $A$. Given an ordinary derivation $d$ on $A$, it is 
easy to verify that the sequence $\{D_n:=d^n/n!\}_{n\geq 0}$ satisfies the following identity:
\begin{equation}
D_n(a\cdot a') =\sum_{i=0}^nD_i(a)\cdot D_{n-i}(a') \qquad \forall\textrm{ } n\geq 0, \textrm{ }a,a'\in A
\end{equation} More generally, we have the notion of a higher (or Hasse-Schmidt) derivation on $A$. 

\medskip
\begin{defn} (see, for instance, \cite{Mat86}) Let $A$ be a commutative algebra over $\mathbb C$. A sequence $D=\{D_n 
\}_{n\geq 0}$ of $\mathbb C$-linear maps on $A$ is said to be a higher (or Hasse-Schmidt) derivation on $A$ if it satisfies:
\begin{equation}\label{3.2}
D_n(a\cdot a') =\sum_{i=0}^nD_i(a)\cdot D_{n-i}(a') \qquad \forall\textrm{ } n\geq 0, \textrm{ }a,a'\in A
\end{equation}

\end{defn}

\medskip
In this paper, we will only work with higher derivations $D=\{D_n\}_{n\geq 0}$ that are normalized, i.e., those higher
derivations $D=\{D_n\}_{n\geq 0}$ which satisfy $D_0=1$. For a normalized higher derivation $D=\{D_n\}_{
n\geq 0}$ it is easy to verify from relation \eqref{3.2} that $D_n(1)=0$ for all $n>0$. 
For more on the structure of higher derivations on an algebra, we refer the reader to \cite{Mirza}, \cite{RoyS} and 
\cite{RMS}. For a higher derivation on $A$, we have already described in \cite{AB} the corresponding Lie derivative on the ordinary Hochschild homology; we are now ready to introduce the action of a higher derivation
on the $q$-Hochschild homology groups of order $Y$ of the algebra $A$. 

\medskip
\begin{lem}\label{lem3.3} Let $A$ be a commutative algebra over $\mathbb C$ and let $D=\{D_n\}_{n\geq 0}$ be a (normalized) higher
derivation on $A$. Then, for any given $k\geq 0$, the higher derivation $D$ induces an  endomorphism $L_D^{k}:
\mathcal L(A)
\longrightarrow \mathcal L(A)$ of the functor $\mathcal L(A):Fin_*\longrightarrow Vect$.
\end{lem} 

\begin{proof} It suffices to prove that for each $k\geq 0$, we have an endomorphism $L_D^k:\mathcal L(A)
\longrightarrow \mathcal L(A)$ of the functor $\mathcal L(A):Fin_*\longrightarrow Vect$ restricted to the subcategory 
$\Gamma$ of $Fin_*$. Given the higher derivation $D=\{D_n\}_{n\geq 0}$ and 
the integer $k\geq 0$, we define morphisms ($\forall$ $n\geq 0$)
\begin{equation}\label{3.3}
\begin{array}{c}
L_D^k([n]):\mathcal L(A)([n])\longrightarrow \mathcal L(A)([n]) \\
(a_0\otimes a_1\otimes ...\otimes a_n)\mapsto \underset{\tiny \begin{array}{c} (p_0,p_1,...,p_n)\\ p_0+p_1+...+p_n=k\\  \end{array}}{\sum}
(D_{p_0}(a_0)\otimes D_{p_1}(a_1)\otimes ....\otimes D_{p_n}(a_n))\\
\end{array}
\end{equation} For the sake of convenience, we will often denote a sum as in \eqref{3.3} taken over all ordered
tuples $(p_0,p_1,...,p_n)$ of non-negative integers such that
$p_0+p_1+...+p_n=k$ simply as 
\begin{equation}\label{3.4}
(a_0\otimes a_1\otimes ...\otimes a_n)\mapsto \underset{\tiny  p_0+p_1+...+p_n=k}{\sum}
(D_{p_0}(a_0)\otimes D_{p_1}(a_1)\otimes ....\otimes D_{p_n}(a_n))
\end{equation}  Let $\phi:[m]\longrightarrow [n]$  be a morphism in $\Gamma$. We let $N(j)$ denote the cardinality
of the set $\phi^{-1}(j)\subseteq [m]$ for any $0\leq j\leq n$. Then, we have, for any
$(a_0\otimes a_1\otimes ... \otimes a_m)\in A\otimes A^{\otimes m}$:
\begin{equation}\label{3.5}
\begin{array}{ll}
L_D^k([n])\circ \mathcal L(A)(\phi)(a_0\otimes a_1\otimes ... \otimes a_m)  &= L_D^k([n])\left(\underset{j=0}{\overset{n}{\bigotimes}}
\textrm{ }\underset{\phi(i)=j}{\prod}a_i\right)
\\ 
&= \underset{\tiny  p_0+p_1+...+p_n=k}{\sum} \left(\underset{j=0}{\overset{n}{\bigotimes}}
\textrm{ }D_{p_j}\left(\underset{\phi(i)=j}{\prod}a_i\right)\right)
\\ 
&= \underset{\tiny  p_0+p_1+...+p_n=k}{\sum} \left(\underset{j=0}{\overset{n}{\bigotimes}}\textrm{ }\textrm{ }\underset{q_1+...+q_{N(j)}=p_j}{\sum}\textrm{ } \underset{\phi(i)=j}{\prod}D_{q_i}(a_i) \right) \\
&= \underset{\tiny  r_0+r_1+...+r_m=k}{\sum} \left( \underset{j=0}{\overset{n}{\bigotimes}}
\textrm{ }\underset{\phi(i)=j}{\prod}D_{r_i}(a_i)\right) \\
&= \underset{\tiny  r_0+r_1+...+r_m=k}{\sum} \mathcal L(A)(\phi) \left( \underset{i=0}{\overset{m}{\bigotimes}}
\textrm{ }D_{r_i}(a_i)\right)\\
& =\mathcal L(A)(\phi)\circ L_D^k([m])(a_0\otimes a_1\otimes ... \otimes a_m) \\
\end{array}
\end{equation} From \eqref{3.5}, it follows that for each $k\geq 0$, $L_D^k:\mathcal L(A)
\longrightarrow \mathcal L(A)$ is an endomorphism of the functor $\mathcal L(A)$ restricted
to $\Gamma$ and hence, taking colimits as in the proof of Lemma \ref{lem1},  $L_D^k$ induces an endomorphism of the functor
$\mathcal L(A):Fin_* \longrightarrow Vect$.

\end{proof}

\medskip

\begin{thm}\label{prop3.3} Let $q\in \mathbb C$ be a primitive $N$-th root of unity. Let $A$ be a commutative algebra over
$\mathbb C$ and let $Y$ be a pointed simplicial finite set. Then, given a higher derivation $D=\{D_n\}_{n\geq 0}$ on $A$, for each $k\geq 0$, 
we have an induced morphism:
\begin{equation}
L_D^{Y,k}:\leftsub{q}{HH}_*^Y(A)=\underset{n=0}{\overset{\infty}{\bigoplus}}\textrm{ }\leftsub{q}{HH}_n^Y(A)\longrightarrow 
\leftsub{q}{HH}_*^Y(A) =\underset{n=0}{\overset{\infty}{\bigoplus}}\textrm{ }\leftsub{q}{HH}_n^Y(A)
\end{equation} on the $q$-Hochschild homology groups of $A$ of order $Y$. 
\end{thm} 

\begin{proof} From Lemma \ref{lem3.3}, we know that for any $k\geq 0$, we have an endomorphism $L_D^k:
\mathcal L(A)\longrightarrow \mathcal L(A)$ of the functor $\mathcal L(A):Fin_*\longrightarrow Vect$. Composing
with the functor $Y:\Delta^{op}\longrightarrow Fin_*$ corresponding to the pointed simplicial finite set $Y$, we have an induced
endomorphism $L_D^{Y,k}:\mathcal L^Y(A)\longrightarrow \mathcal L^Y(A)$ of the functor 
$\mathcal L^Y(A)=\mathcal L(A)\circ Y:\Delta^{op}\overset{Y}{\longrightarrow}Fin_* \overset{\mathcal L(A)}{
\longrightarrow}Vect$. Accordingly, $L_D^{Y,k}$ induces an endomorphism on the homology objects of the $N$-complex $(\mathcal L^Y(A),\leftsub{q}{b})$ 
associated to the simplicial vector space  $\mathcal L^Y(A)$ as in \eqref{2.5cxw}. Hence, we have induced morphisms
$L_D^{Y,k}:\leftsub{q}{HH}_*^Y(A)\longrightarrow \leftsub{q}{HH}^Y_*(A)$ on the $q$-Hochschild homology groups
of order $Y$. 

\end{proof}

\medskip
We have already shown in the last section that $\leftsub{q}{HH}_*^Y(A)$ is a left module over the universal enveloping
algebra $\mathcal H=\mathcal U(Der(A))$ of the Lie algebra of derivations on $A$. 
 Given a higher derivation $D=\{D_k\}_{k\geq 0}$ on a $\mathbb C$-algebra $A$, Mirzavaziri \cite{Mirza} has shown that the higher derivation $D$ may be expressed as
follows: there exists a sequence of ordinary derivations $\{d_n\}_{n\geq 0}$, $d_n\in Der(A)$ such that:
\begin{equation}\label{3.7}
D_k=\sum_{i=1}^k\left(\underset{\sum_{j=1}^ir_j=k}{\sum}\left(\underset{j=1}{\overset{i}{\prod}}\frac{1}{r_j+...+r_i}\right)d_{r_1}....d_{r_i}\right)
\end{equation} From \eqref{3.7}, it is clear that given a higher derivation $D=\{D_k\}_{k\geq 0}$ on $A$, each $D_k$
is an element of the Hopf algebra $\mathcal H=\mathcal U(Der(A))$. Hence,  it follows from Proposition \ref{prop2.6} that
each operator $D_k\in \mathcal H$ induces a morphism $L_{D_k}^Y:\leftsub{q}{HH}_*^Y(A)
\longrightarrow \leftsub{q}{HH}_*^Y(A)$ on the $q$-Hochschild homology groups of order $Y$.  We will now show that
the morphisms $L_{D_k}^Y$, $k\geq 1$ are identical to the morphisms $L_D^{Y,k}: \leftsub{q}{HH}_*^Y(A)
\longrightarrow \leftsub{q}{HH}_*^Y(A)$ described in Proposition \ref{prop3.3}.

\medskip
\begin{thm}  Let $q\in \mathbb C$ be a primitive $N$-th root of unity. Let $A$ be a commutative algebra over
$\mathbb C$ and let $Y$ be a pointed simplicial finite set. Let $D=\{D_k\}_{k\geq 0}$ denote a higher derivation on $A$. For 
any $k\geq 1$, let $L_{D_k}^Y:\leftsub{q}{HH}_*^Y(A)
\longrightarrow \leftsub{q}{HH}_*^Y(A)$ be the morphism induced by $D_k\in \mathcal H$ as in Proposition
\ref{prop2.6} and let $L_D^{Y,k}: \leftsub{q}{HH}_*^Y(A)
\longrightarrow \leftsub{q}{HH}_*^Y(A)$ be the morphism induced by $D$ as in Proposition \ref{prop3.3}. Then, we have
$L_{D_k}^Y= L_{D}^{Y,k}: \leftsub{q}{HH}_*^Y(A)
\longrightarrow \leftsub{q}{HH}_*^Y(A)$. 
\end{thm}

\begin{proof} From the proofs of Lemma \ref{lem1} and Lemma \ref{lemKS}, it follows that the element $D_k\in \mathcal H=\mathcal U(Der(A))$ of the universal
enveloping algebra $\mathcal H$ defines an endomorphism $L_{D_k}:\mathcal L(A)\longrightarrow 
\mathcal L(A)$ of the functor $\mathcal L(A):Fin_*\longrightarrow Vect$. From the proofs of Proposition \ref{prop2.4r} and Proposition \ref{prop2.6}, it
is clear that the morphism $L_{D_k}^Y:\leftsub{q}{HH}_*^Y(A)
\longrightarrow \leftsub{q}{HH}_*^Y(A)$ is obtained from the endomorphism $L_{D_k}^Y:
\mathcal L^Y(A)\longrightarrow \mathcal L^Y(A)$ of the functor $\mathcal L^Y(A)=\mathcal L(A)\circ Y$ induced by
$L_{D_k}:\mathcal L(A)\longrightarrow \mathcal L(A)$. 

\medskip
Similarly, from Lemma \ref{lem3.3}, it follows that the higher derivation $D$ induces an endomorphism $L_D^k:
\mathcal L(A)\longrightarrow \mathcal L(A)$ of the functor $\mathcal L(A):Fin_*\longrightarrow Vect$. From the proof of Proposition \ref{prop3.3}, it follows that the morphism $L_D^{Y,k}: \leftsub{q}{HH}_*^Y(A)
\longrightarrow \leftsub{q}{HH}_*^Y(A)$ is obtained from the endomorphism $L_D^{Y,k}:\mathcal L^Y(A)
\longrightarrow \mathcal L^Y(A)$ of the functor $\mathcal L^Y(A)=\mathcal L(A)\circ Y$ induced by
$L_D^{k}:\mathcal L(A)\longrightarrow \mathcal L(A)$. Hence, in order to prove the result, we need to show that
$L_D^k=L_{D_k}$ as endomorphisms of the functor $\mathcal L(A):Fin_*\longrightarrow Vect$. As before, it suffices
to show that $L_D^k=L_{D_k}$ as  endomorphisms of the functor $\mathcal L(A)$ restricted to the subcategory
$\Gamma$ of $Fin_*$.

\medskip Let $\Delta:
\mathcal H\longrightarrow \mathcal H\otimes \mathcal H$ denote
the coproduct on $\mathcal H$. For any $h\in \mathcal H$ and any $n\geq 0$, we write $\Delta^n(h)=\sum h_{(1)}\otimes 
h_{(2)}\otimes ... \otimes h_{(n+1)}$. Then, we have an induced endomorphism $L_h:
\mathcal L(A)\longrightarrow \mathcal L(A)$ of the functor $\mathcal L(A):Fin_*
\longrightarrow Vect$. Further, we note that the equation
\begin{equation}
L_h([n])(a_0\otimes a_1\otimes ... \otimes a_n)=\sum (h_{(1)}(a_0)\otimes h_{(2)}(a_1)\otimes ...h_{(n+1)}(a_n))\quad
\forall (a_0\otimes ... \otimes a_n)\in \mathcal L(A)([n])\
\end{equation} holds for all $h\in Der(A)\subseteq \mathcal H$ and hence for all $h\in \mathcal H=\mathcal U(Der(A))$. 
From the definition of $L_D^k$ in Lemma \ref{lem3.3}, we now see that in order to show that 
$L_D^k=L_{D_k}$, it suffices to show that
\begin{equation}\label{3.9}
\Delta^n(D_k)=\sum_{\sum_{i=0}^np_i=k}D_{p_0}\otimes D_{p_1}\otimes ... \otimes D_{p_n} \qquad \forall\textrm{ } n\geq 0
\end{equation} We will prove \eqref{3.9} by induction on $k$. For any given $n\geq 0$, it is clear that the equation 
\eqref{3.9} holds for $k=0$ and $k=1$. We now suppose that its holds for any $0\leq k\leq K$. From \cite[Proposition 2.1]{Mirza}, we know that
\begin{equation}\label{3.10}
D_{M+1}=\frac{1}{M+1}\sum_{m=0}^Md_{m+1}D_{M-m}\qquad \forall\textrm{ }M\geq 0
\end{equation} where the $d_{m+1}$ are the derivations corresponding to the higher derivation $D=\{D_n\}_{n\geq 0}$
as described in \eqref{3.7}. From \eqref{3.10}, it follows that $\Delta^n(D_{K+1})=\frac{1}{K+1}\sum_{m=0}^K\Delta^n(d_{m+1})\Delta^n(D_{K-m})$ and hence 
\begin{equation}\label{3.11}
\Delta^n(D_{K+1})=
\frac{1}{K+1}\sum_{m=0}^K\left(\sum_{j=0}^nd_{m+1}^j\right)\left(\sum_{\sum_{i=0}^np_i=K-m}D_{p_0}\otimes 
D_{p_1}\otimes ...\otimes D_{p_n}\right)
\end{equation} where $d^j_{m+1}$ denotes the term $1\otimes 1\otimes ...\otimes d_{m+1} \otimes ... \otimes 1$ (i.e., $d_{m+1}$ at the $j$-th position) appearing in the expression for $\Delta^n(d_{m+1})$. We now consider ordered
tuples $(p'_0,p'_1,...,p'_n)$ of non-negative integers such that $p'_0+p'_1+....+p'_n=K+1$. Then, we can write:
\begin{equation}\label{3.12}
\begin{array}{l}
\underset{m=0}{\overset{K}{\sum}}\left(\underset{j=0}{\overset{n}{\sum}}d_{m+1}^j\right)\left(\underset{\sum_{i=0}^np_i=K-m}{\sum}D_{p_0}\otimes 
D_{p_1}\otimes ...\otimes D_{p_n}\right)\\
=\underset{\sum_{i=0}^np'_i=K+1}{\sum}\textrm{ }\underset{j=0,p'_j\geq 1}{\overset{n}{\sum}}\textrm{ }
\underset{m=0}{\overset{p'_j-1}{\sum}}\textrm{ }d^j_{m+1}\cdot (D_{p'_0}\otimes ... \otimes D_{p'_j-m-1}\otimes ... 
\otimes D_{p'_n}) \\
= \underset{\sum_{i=0}^np'_i=K+1}{\sum}\textrm{ }\underset{j=0,p'_j\geq 1}{\overset{n}{\sum}}\textrm{ }
\underset{m=0}{\overset{p'_j-1}{\sum}}\textrm{ }(D_{p'_0}\otimes ... \otimes d_{m+1} D_{p'_j-m-1}\otimes ... 
\otimes D_{p'_n}) \\
\end{array}
\end{equation}  From \eqref{3.10}, it follows that $\sum_{m=0}^{p'_j-1}d_{m+1}D_{p'_j-m-1}=p'_j\cdot D_{p'_j}$ and hence:
\begin{equation}\label{3.13}
\underset{m=0}{\overset{p'_j-1}{\sum}}\textrm{ }(D_{p'_0}\otimes ... \otimes d_{m+1} D_{p'_j-m-1}\otimes ... 
\otimes D_{p'_n}) = p'_j\cdot (D_{p'_0}\otimes ... \otimes D_{p'_j}\otimes ... 
\otimes D_{p'_n}) 
\end{equation} Combining \eqref{3.11}, \eqref{3.12} and \eqref{3.13}, it follows that:
\begin{equation}
\begin{array}{ll}
\Delta^n(D_{K+1})&=\frac{1}{K+1} \left(\underset{\sum_{i=0}^np'_i=K+1}{\sum}\textrm{ }\underset{j=0,p'_j\geq 1}{\overset{n}{\sum}}\textrm{ }
 p'_j\cdot (D_{p'_0}\otimes ... \otimes D_{p'_j}\otimes ... 
\otimes D_{p'_n})\right) \\ 
& = \frac{1}{K+1} \left(\underset{\sum_{i=0}^np'_i=K+1}{\sum}\textrm{ }(K+1)
 \cdot (D_{p'_0}\otimes ... \otimes D_{p'_j}\otimes ... 
\otimes D_{p'_n})\right) \\ 
& = \underset{\sum_{i=0}^np'_i=K+1}{\sum}\textrm{ }
 (D_{p'_0}\otimes ... \otimes D_{p'_j}\otimes ... 
\otimes D_{p'_n}) \\ 
\end{array}
\end{equation} This proves the result of \eqref{3.9} for $K+1$. 

\end{proof}

 \medskip
\section{Action on bivariant $q$-Hochschild cohomology groups}

\medskip
Let $A$ be a commutative algebra over $\mathbb C$ and let $q\in \mathbb C$ be a primitive $N$-th root of unity. Let $Y$ be a pointed simplicial finite set. In this section, we will define the bivariant $q$-Hochschild cohomology
groups $\{HH_Y^n(A,A)\}_{n\in \mathbb Z}$ of $A$ of order $Y$ and show that a derivation $D$ on $A$ induces a morphism $\underline{L}_D^{Y,n}(A,A):\leftsub{q}{HH}^n_Y(A,A)
\longrightarrow \leftsub{q}{HH}^n_Y(A,A)$. For the ordinary bivariant Hochschild cohomology groups $\{HH^n(A,A)\}_{n\in 
\mathbb Z}$, we have already studied this morphism in \cite{AB2}. For the definition and properties
of ordinary bivariant Hochschild cohomology, we refer the reader to \cite[$\S$ 5.1]{Lod} (see also
the original paper of Jones and Kassel \cite{JK}). We start by defining the bivariant $q$-Hochschild cohomology
groups of order $Y$. 

\medskip
\begin{defn}\label{Df4.1} Let $(\mathcal L^Y(A),\leftsub{q}{b})$ be the $N$-complex corresponding to the simplicial vector space
$\mathcal L^Y(A)$ as defined  in \eqref{2.5cxw}. We consider the $q$-Hom complex
$\underline{Hom}(\mathcal L^Y(A),\mathcal L^Y(A))$ of these $N$-complexes which is defined as follows:
\begin{equation}\label{4.1}
\underline{Hom}(\mathcal L^Y(A),\mathcal L^Y(A))_n := \underset{i\in \mathbb Z}{\prod}\textrm{ }Hom_{Vect}(\mathcal L^Y(A)_i,
\mathcal L^Y(A)_{i+n})  
\end{equation} Further, if the family $f=\{f_i:\mathcal L^Y(A)_i\longrightarrow \mathcal L^Y(A)_{i+n}\}_{i\in \mathbb Z}$ is an element
of $\underline{Hom}(\mathcal L^Y(A),\mathcal L^Y(A))_n$, then the differential 
$\leftsub{q}{\partial}_n:\underline{Hom}(\mathcal L^Y(A),\mathcal L^Y(A))_n\longrightarrow \underline{Hom}(\mathcal L^Y(A),\mathcal L^Y(A))_{n-1}$ is defined by setting:
\begin{equation}\label{4.2}
\begin{array}{c}
\leftsub{q}{\partial}_n(f):=\{\leftsub{q}{\partial}_n(f)_i:\mathcal L^Y(A)_i\longrightarrow \mathcal L^Y(A)_{i+n-1}\}_{i
\in \mathbb Z} \\
\leftsub{q}{\partial}_n(f)_i= \leftsub{q}{b}_{i+n}\circ f_i - q^n f_{i-1}\circ \leftsub{q}{b}_i
\end{array}
\end{equation} For any given $n\in \mathbb Z$, we define the bivariant $q$-Hochschild cohomology group $\leftsub{q}{HH}^n_Y(A,A)$ of $A$ of order $Y$ to be the homology object
\begin{equation}\label{4.3}
\leftsub{q}{HH}^n_Y(A,A):=H_{\{-n\}}(\underline{Hom}(\mathcal L^Y(A),\mathcal L^Y(A)), \leftsub{q}{\partial})
\end{equation} of the $N$-complex $(\underline{Hom}(\mathcal L^Y(A),\mathcal L^Y(A)),\leftsub{q}{\partial})$.

\end{defn}

\medskip
We mention that it follows from \cite[Proposition 1.8]{Kap} that the $q$-Hom complex $(\underline{Hom}(\mathcal L^Y(A),\mathcal L^Y(A)), \leftsub{q}{\partial})$ as defined in \eqref{4.1} and \eqref{4.2} is also an $N$-complex. We now make
the convention that if $M=\oplus_{i\in \mathbb Z}M_i$ is a graded vector space and $f=\{f_i:M_i\longrightarrow M_{i+m}\}_{i\in \mathbb Z}$ and $g=\{g_i:M_i
\longrightarrow M_{i+n}\}_{i\in \mathbb Z}$ are two morphisms of homogenous degree $m$ and $n$ respectively, we will
write $[f,g]:=f\circ g-q^{mn}g\circ f$ for their graded $q$-commutator. 

\medskip
\begin{lem}\label{lem4.2} Let $L^m=\{L^m_i\}_{i\in \mathbb Z}$ denote a collection of maps $L^m_i:\mathcal L^Y(A)_i
\longrightarrow \mathcal L^Y(A)_{i+m}$. Given an element $f=\{f_i\}_{i\in \mathbb Z}$ in $\underline{Hom}(\mathcal L^Y(A),\mathcal L^Y(A))_n$, we define $\underline{L}^m(f)\in \underline{Hom}(\mathcal L^Y(A),\mathcal L^Y(A))_{m+n}$ by setting:
\begin{equation}
\underline{L}^m(f)_i:\mathcal L^Y(A)_i\longrightarrow \mathcal L^Y(A)_{i+m+n}\qquad \underline{L}^m(f)_i:=L^m_{i+n}\circ f_i - q^{mn}f_{i+m}\circ L^m_i 
\end{equation} Then, if $q^{2m}=1$, the endomorphism $\underline{L}^m: \underline{Hom}(\mathcal L^Y(A),\mathcal L^Y(A))
\longrightarrow \underline{Hom}(\mathcal L^Y(A),\mathcal L^Y(A))$ of homogeneous degree $m$ satisfies the following relation:
\begin{equation}
[\leftsub{q}{\partial},\underline{L}^m](f)= [\leftsub{q}{b},L^m]f + q^{mn+m+n}f[L^m,\leftsub{q}{b}] \qquad
\forall\textrm{ }f\in \underline{Hom}(\mathcal L^Y(A),\mathcal L^Y(A))_n, n\in \mathbb Z
\end{equation} 
\end{lem}

\begin{proof} We consider:
\begin{equation}\label{4.6}
\begin{array}{ll}
((\leftsub{q}{\partial}\circ \underline{L}^m)(f))_i& = \leftsub{q}{b}_{i+m+n}\circ \underline{L}^m(f)_i- q^{m+n}\underline{L}^m(f)_{i-1}\circ \leftsub{q}{b}_i\\
& = \leftsub{q}{b}_{i+m+n}\circ L_{i+n}^m\circ f_i - q^{mn}  \leftsub{q}{b}_{i+m+n}\circ f_{i+m}\circ L_i^m \\
& \textrm{ }\textrm{ } - q^{m+n}L^m_{i+n-1}\circ f_{i-1}\circ \leftsub{q}{b}_i+ q^{mn+m+n}f_{i+m-1}\circ L_{i-1}^m
\circ \leftsub{q}{b}_i\\ 
((\underline{L}^m\circ \leftsub{q}{\partial})(f))_i & = L^m_{i+n-1}\circ \leftsub{q}{\partial}(f)_i - q^{m(n-1)}\leftsub{q}{\partial}(f)_{i+m}\circ L_i^m \\
& = L^m_{i+n-1}\circ \leftsub{q}{b}_{i+n}\circ f_i - q^nL^m_{i+n-1}\circ f_{i-1}\circ \leftsub{q}{b}_i  \\
& \textrm{ }\textrm{ }- q^{m(n-1)}\leftsub{q}{b}_{i+m+n}\circ f_{i+m}\circ L^m_i + q^{mn-m+n}f_{i+m-1}\circ\leftsub{q}{b}_{i+m}\circ L_i^m \\
\end{array}
\end{equation}
From \eqref{4.6}, it follows that:
\begin{equation*}
\begin{array}{l}
([\leftsub{q}{\partial},\underline{L}^m](f))_i =((\leftsub{q}{\partial}\circ \underline{L}^m)(f))_i - 
q^{-m}((\underline{L}^m\circ \leftsub{q}{\partial})(f))_i  \\
 = (\leftsub{q}{b}_{i+m+n}\circ L_{i+n}^m - q^{-m}L^m_{i+n-1}\circ \leftsub{q}{b}_{i+n})\circ f_i 
+f_{i+m-1}\circ q^{mn+m+n}(L^m_{i-1}\circ \leftsub{q}{b}_i-q^{-2m}(q^{-m}\leftsub{q}{b}_{i+m}\circ L_i^m))\\
-q^{mn}(1-q^{-2m})\leftsub{q}{b}_{i+m+n}\circ f_{i+m}\circ L^m_i
-q^{m+n}(1-q^{-2m})L^m_{i+n-1}\circ f_{i-1}\circ \leftsub{q}{b}_i\\
\end{array}
\end{equation*}
Combining with the fact that $q^{2m}=1$, it follows from the above expression that:
\begin{equation}
[\leftsub{q}{\partial},\underline{L}^m](f)= [\leftsub{q}{b},L^m]f + q^{mn+m+n}f[L^m,\leftsub{q}{b}] \\
\end{equation}
\end{proof}

\medskip
\begin{thm} Let $q\in \mathbb C$ be a primitive $N$-th root of unity. Let $A$ be a commutative algebra
over $\mathbb C$ and let $D:A\longrightarrow A$ be a derivation on $A$. Let $Y$ be a pointed simplicial finite set. Then, for each
$n\in \mathbb Z$, the derivation $D$ on $A$ induces a morphism
\begin{equation}
\underline{L}_D^{Y,n}: \leftsub{q}{HH}_Y^n(A,A)\longrightarrow  \leftsub{q}{HH}_Y^n(A,A)
\end{equation} on the bivariant $q$-Hochschild cohomology groups of order $Y$. 
\end{thm}

\begin{proof} From the proof of Proposition \ref{prop2.6}, we know that the derivation $D$ induces an endomorphism
$L_D^Y:\mathcal L^Y(A)\longrightarrow \mathcal L^Y(A)$ of the simplicial vector space $\mathcal L^Y(A)$. Accordingly, we have a collection
of maps $L_D^Y=\{L_{D,i}^{Y}:\mathcal L^Y(A)_i\longrightarrow \mathcal L^Y(A)_i\}_{i\in \mathbb Z}$ determined by
the endomorphism $L_D^Y$. Applying Lemma \ref{lem4.2} with $m=0$ (and hence $q^{2m}=1$), it follows that $L_D^Y$
determines a morphism 
\begin{equation}
\underline{L}_D^Y: \underline{Hom}(\mathcal L^Y(A),\mathcal L^Y(A))\longrightarrow \underline{Hom}(
\mathcal L^Y(A),\mathcal L^Y(A))
\end{equation} of homogeneous degree $m=0$ satisfying:
\begin{equation}\label{4.10}
[\leftsub{q}{\partial},\underline{L}^Y_D](f)= [\leftsub{q}{b},L^Y_D]f + q^{n}f[L^Y_D,\leftsub{q}{b}]\qquad 
\forall\textrm{ }f\in \underline{Hom}(\mathcal L^Y(A),\mathcal L^Y(A))_n, n\in \mathbb Z
\end{equation}
Again, since $L_D^Y:\mathcal L^Y(A)\longrightarrow \mathcal L^Y(A)$ is a morphism of simplicial vector spaces, 
the morphisms $\{L_{D,i}^{Y}:\mathcal L^Y(A)_i\longrightarrow \mathcal L^Y(A)_i\}_{i\in \mathbb Z}$ commute with the face maps $d_i^j:\mathcal L^Y(A)_i\longrightarrow \mathcal L^Y(A)_{i-1}$, $0\leq j\leq i$, $i\geq 0$ of the simplicial
vector space $\mathcal L^Y(A)$. By definition, $\leftsub{q}{b}_i:=\sum_{j=0}^iq^jd_i^j$ and hence we have:
\begin{equation}
[\leftsub{q}{b},L^Y_D]=[L^Y_D,\leftsub{q}{b}]=0 
\end{equation} Applying this to \eqref{4.10}, it follows that:
\begin{equation}\label{4.12}
[\leftsub{q}{\partial},\underline{L}^Y_D]=\leftsub{q}{\partial}\circ \underline{L}^Y_D-q^{-m}\underline{L}^Y_D
\circ \leftsub{q}{\partial}=\leftsub{q}{\partial}\circ \underline{L}^Y_D-\underline{L}^Y_D
\circ \leftsub{q}{\partial}=0
\end{equation} From \eqref{4.12}, it follows that the endomorphism $\underline{L}_D^Y: \underline{Hom}(\mathcal L^Y(A),\mathcal L^Y(A))\longrightarrow \underline{Hom}(
\mathcal L^Y(A),\mathcal L^Y(A))$ of degree zero commutes with the differential $\leftsub{q}{\partial}$ on the $N$-complex
$\underline{Hom}(\mathcal L^Y(A),\mathcal L^Y(A))$. This induces morphisms ($\forall$ $n\in \mathbb Z$):
\begin{equation}
\begin{CD}
 \leftsub{q}{HH}_Y^n(A,A)=H_{\{-n\}}(\underline{Hom}(\mathcal L^Y(A),\mathcal L^Y(A)), \leftsub{q}{\partial})\\ @V\underline{L}_D^{Y,n}VV \\ \leftsub{q}{HH}_Y^n(A,A)=H_{\{-n\}}(\underline{Hom}(\mathcal L^Y(A),\mathcal L^Y(A)), \leftsub{q}{\partial}) \\
\end{CD}
\end{equation} on the bivariant $q$-Hochschild cohomology groups of order $Y$. 
\end{proof}

\end{document}